\documentclass[reqno, english, 10pt]{amsart}
\usepackage[T1]{fontenc}
\usepackage[latin9]{inputenc}
\usepackage{amsthm}
\usepackage{amssymb}

\makeatletter
\numberwithin{equation}{section}
\numberwithin{figure}{section}
  \theoremstyle{plain}
  \newtheorem*{thm*}{\protect\theoremname}
\theoremstyle{plain}
\newtheorem{thm}{\protect\theoremname}
  \theoremstyle{plain}
  \newtheorem{lem}[thm]{\protect\lemmaname}
  \theoremstyle{plain}
  \newtheorem{cor}[thm]{\protect\corollaryname}
\theoremstyle{remark}
\newtheorem{rem}[thm]{Remark}

\makeatother

\usepackage{babel}
  \providecommand{\corollaryname}{Corollary}
  \providecommand{\lemmaname}{Lemma}
  \providecommand{\theoremname}{Theorem}
\providecommand{\theoremname}{Theorem}

\begin{document}
\global\long\def\l{\lambda}
\global\long\def\ep{\epsilon}

\title{A Blow-Up Result for Dyadic Models of the Euler Equations}

\author{In-Jee Jeong and Dong Li}


\email{ijeong@math.princeton.edu, dli@math.ubc.ca}
\begin{abstract}
We partially answer a question raised by Kiselev and Zlatos in
\cite{MR2180809}; in the generalized dyadic model of the Euler
equation, a blow-up of $H^{1/3+\delta}$-norm occurs. We recover a
few previous blow-up results for various related dyadic models as
corollaries.
\end{abstract}
\maketitle

\section{Introduction}

In this paper, we consider the following infinite system of ordinary
differential equations (ODEs):
\begin{align}
\frac{da_{j}(t)}{dt}&
=\alpha\big(\lambda^{j}a_{j-1}^{2}(t)-\l^{j+1}a_{j}(t)a_{j+1}(t)\big)
\notag\\
& \qquad +\beta\big(\l^{j}a_{j-1}(t)a_{j}(t)-\l^{j+1}a_{j+1}^{2}(t)\big),\label{eq:KPO}
\end{align}
for $j\geq0$ and with the boundary condition $a_{-1}(t)\equiv0$.
The coefficients $\alpha$ and $\beta$ are usually taken to be nonnegative constants.
We will assume $\l=2$ throughout, but our results hold for arbitrary
$\l>1$ with proper adjustments of the parameters. The special case
$\alpha=1,\beta=0$ is often called the KP equations, which have appeared
in the literature almost simultaneously in two papers \cite{MR2038114,MR2095627}.
The opposite extreme $\alpha=0,\beta=1$  first appeared in Obukhov's
work \cite{Obukhov1971} and was suggested as an alternative to the
KP equations in \cite{MR2231615}. Hence (\ref{eq:KPO}) can be viewed
as a linear combination of these two models.

These types of infinite system of ODEs are called dyadic models of
the Euler equations. For a heuristic derivation of the KP equations
from the Euler equations, one can see \cite{MR2095627}
for an argument based on the wavelet expansion of a scalar function over dyadic
cubes. Alternatively, consider the Euler equations in $\mathbb{R}^{n}$ with periodic
boundary conditions and rewrite the equations in terms of the Fourier
coefficients of the velocity vector field. Then one obtains an infinite
system of ODEs for the evolution of Fourier coefficients which share
several structural similarities with (\ref{eq:KPO}). We will return
to this point after Lemma \ref{lem:KPO}, from which equations (\ref{eq:KPO})
appear naturally from some constitutive relations.

To state blow-up and regularity results for dyadic models, let us
first define analogues of the Sobolev norms in the space of sequences.
The $H^{s}$-norm of a solution $a=(a_{0},a_{1},...)$ at time $t$
is defined by the formula
\[
||a(t)||_{s}^{2}:=\sum_{j=0}^{\infty}2^{2sj}a_{j}^{2}(t).
\]
In particular, we define the energy $E(t)$ as the square of $H^{0}$-norm (or the usual $l^2$-norm):
\[
E(t):=\sum_{j=0}^{\infty}a_{j}^{2}(t).
\]
Regarding the KP equations, the following blow-up result have been
proved several times (we have listed the references in a more or
less chronological order):
\begin{thm*}[\cite{MR2038114,MR2095627,MR2231615,MR2180809,MR2746670}]
 For every nonzero initial data, the $H^{s}$-norm of any solution\footnote{The meaning of solution
 here requires some clarification. It is known (cf. Proposition 2.1 of \cite{MR2180809})
 that local wellposedness holds in $C_t^0 H^s$ with $s\ge 1$. Alternatively one can work with  finite energy
 Leray-Hopf type solutions (cf. \cite{MR3057168,MR2415066}) for which uniqueness is a subtle issue.}
becomes infinite in finite time for all $s>s_{cr}:=1/3$.
\end{thm*}
Now for the Obukhov equations, there is the following regularity result:
\begin{thm*}[\cite{MR2180809}]
 If the initial data have finite $H^{s}$-norm for some $s>1$, then
the corresponding solution exists globally and has finite $H^{s}$-norm for all $t\geq0$.
\end{thm*}
In \cite{MR2180809}, Kiselev and Zlatos raised the question of whether blow-up
can occur in the case $\alpha,\beta>0$ (or more generally $\operatorname{sgn}(\alpha)=
\operatorname{sgn}(\beta)$). Corollary \ref{cor:KPO} of this paper
answers this question affirmatively, at least in the case when $\beta$ is
small relative to $\alpha$. Our proof of blow-up is quite different
and seems to be simpler than the previous proofs for the KP equations
(\cite{MR2038114,MR2095627,MR2231615,MR2180809}). Roughly speaking, all the previous blow-up proofs
rest on the intuition that at least a fixed proportion
of the energy contained in the $j$th component must be transferred
to the higher components within a time scale of $\tau^{-j}$ for some
$\tau>0$. To achieve this, one has to make strong use of the ``positivity''
of the KP equations; that is, once we have $a_{j}(t_{0})\geq0$ for
some $j$ and $t_{0}$,  then $a_{j}(t)\geq0$ for all future $t>t_{0}$. This
positivity in turn implies that there is no ``backward'' transfer
of energy; to be more precise, if the initial data satisfy $a_{j_0+1}(0)\geq0$, then
$a_{j_0+1}(t)\ge 0$, and
\[
\frac{d}{dt}E_{j_{0}}(t):=\frac{d}{dt}\big(\sum_{j=0}^{j_{0}}a_{j}^{2}(t)\big)
=-2 \lambda^{j_0+1} a_{j_0}^2 a_{j_0+1}
\leq0
\]
for all  $t\geq0$. Unfortunately, this mechanism of forward energy transfer
 seems to break down once we have both $\alpha,\beta>0$. The proof in
\cite{MR2746670} still makes use of positivity but it appears to be different
from others; we will come back to their proof after Lemma \ref{lem:KPO}.

Next, let us consider the following system of equations, where there
is an extra ``dissipation'' term on the right hand side:

\begin{equation}
\frac{da_{j}}{dt}=\alpha\big(\lambda^{j}a_{j-1}^{2}-\l^{j+1}a_{j}a_{j+1}\big)+
\beta\big(\l^{j}a_{j-1}a_{j}-\l^{j+1}a_{j+1}^{2}\big)-\nu\l^{2\gamma j}a_{j},
\label{eq:NSE}
\end{equation}
again for $j\geq0$ and with $a_{-1}(t)\equiv0$. Here, $\alpha=1,\nu=1$
can be assumed with appropriate rescaling and $\gamma>0$ is a parameter
representing the intensity of the dissipation. In the special
case $\beta=0$, these equations are often called the dyadic Navier-Stokes
Equations (NSEs), and they are already studied quite extensively in
the literature. In particular, the following blow-up result has been
proved by Cheskidov in \cite{MR2415066}:
\begin{thm*}[\cite{MR2415066}]
 Consider equations (\ref{eq:NSE}) with $\alpha=1$, $\beta=0$,
$\nu=1$, and $\gamma<1/3$. For every $\delta>0$, there exists a
constant $M(\delta)$ such that if the initial data satisfy $a_{j}(0)\geq0$
for all $j\geq0$ and $||a(0)||_{\delta}>M(\delta)$, then $||a(t)||_{1/3+\delta}^{3}$
is not locally integrable on $[0,\infty)$.
\end{thm*}
In particular, any solution blows up in finite time in $H^{1/3+\delta}$-norm
for every $\delta>0$. Our proof of the main theorem recovers this
$H^{1/3+\delta}$-norm blow-up in Corollary \ref{cor:D}.

\section{Results and Conjectures}

To begin, we borrow a lemma from \cite{MR2180809} from which
the model (\ref{eq:KPO}) follows naturally. We omit the proof since it is immediate in view of
the (formal) energy conservation constraint.
\begin{lem}
\label{lem:KPO}Assume that real-valued functions $a_{j}(t)$ ($j\geq0$)
satisfy a system of ODEs of the form
\[
\frac{da_{j}(t)}{dt}=F_{j}(a(t))
\]
where
\begin{itemize}
\item for each $j\geq0$, the map $F_{j}$ is a quadratic function of $a(t)$;
\item $F_{j}$ can involve only $a_{j-1}(t)$, $a_{j}(t)$, and $a_{j+1}(t)$;
\item each term for $F_{j}$ has a factor of $\l^{j}$ times a constant
which is independent of $j$, i.e.
\begin{align*}
F_j= \sum_{\substack{\mu_1=\pm 1, 0\\ \mu_2 =\pm 1,0}} C_{\mu_1,\mu_2} \lambda^j a_{j+\mu_1}(t) a_{j+\mu_2}(t),
\end{align*}
where $C_{\mu_1,\mu_2}$ are constants independent of $j$;

\item and the energy $\sum a_{j}^{2}(t)$ is (formally) conserved.
\end{itemize}
Then the system is necessarily of the form (\ref{eq:KPO}).
\end{lem}

The Euler equations are, of course, energy conserving (for smooth
solutions) and have quadratic nonlinearity. The factor $\l^{j}$ was
inserted so that the dyadic model would share similar functional
estimates with the Euler equations. One can argue that to model 3D
Euler equations, the choice $\l=2^{5/2}$ is
appropriate\footnote{Roughly speaking, this is based on the estimate
that (here $P_{2^j}$ is the usual Littlewood-Paley projector adapted
to the frequency block $|\xi|\ \sim 2^{j}$ and one can think of $u$
as the velocity in Euler) $\| P_{2^j} u \cdot \nabla P_{2^j} u
\|_{L^{2}(\mathbb R^d)} \lesssim 2^{(\frac d2 +1)j} \| P_{2^j} u
\|_{L^2(\mathbb R^d)}^2.$ For $d=3$, the factor is $2^{5/2}$. }
 (see \cite{MR2038114}). Lastly, the fact that $F_{j}$
only consists of $a_{j-1}$ and $a_{j+1}$ certainly does not hold
in the case of the Euler equations, but certain ``locality of interactions''
assumptions are believed to hold in the theory of turbulence. For example
one can see \cite{MR1877599, 2005PhyD..207...91E}.

Let us remark on the property of energy conservation. A formal calculation yields
that
\begin{align*}
\frac{d}{dt}E(t) &=2\sum_{j\geq0}a_{j}(\lambda^{j}a_{j-1}^{2}-\l^{j+1}a_{j}a_{j+1})\\
&=2\big(\sum_{j\geq1}\l^{j}a_{j-1}^{2}a_{j})-2(\sum_{j\geq0}\l^{j+1}a_{j}^{2}a_{j+1})=0.
\end{align*}
But in the above computation, an interchange of the order of
summation and differentiation must be justified, and it is
sufficient to require that $||a(t)||_{s}<\infty$ for $s>s_{cr}=1/3$.
However, for solutions with less regularity, this computation is no
longer valid and the dissipation of energy can indeed
occur\footnote{This is in some sense connected to the Onsager's
conjecture.}. In \cite{MR2746670} it was established that for every
initial condition with nonnegative components, the energy dissipates
to zero as $t\rightarrow\infty$. In particular, it implies
finite-time blow-up in every $H^{s}$ norm for $s>1/3$.

We are ready to state our main result.
\begin{thm}[The full model with diffusion]
\label{thm:KPOD}Consider the equations (\ref{eq:NSE}) with parameters
$\l=2$, $\alpha=1$, $\nu=1$, and $\beta\geq0$. For every $s>1/3$
and $\gamma<1/3$, there exists a value $\beta_{s,\gamma}>0$ such that
for $\beta\in[0,\beta_{s,\gamma})$, there exists a class of initial data for which the
corresponding solutions blow up in finite time in $H^{s}$-norm.
More precisely, for each solution $a(t)$,
 $\|a(t)\|_s^2$ is not locally integrable on $[0,\infty)$.
\end{thm}
\begin{rem}
As will be clear from our proof, the initial data $a(0)=(a_j(0))_{j=0}^{\infty}$ can even taken
to be compactly supported, in the sense that for some integer $j_0>0$, $a_j(0)=0$ for all $j\ge j_0$.
\end{rem}

\begin{rem}
Instead of considering only nearest neighborhood interactions (i.e $a_{j}$, $a_{j-1}$, $a_{j+1}$),
one can generalize the full model \eqref{eq:NSE} to arbitrarily finitely many (or even infinitely many with sufficiently fast
decay of interaction) neighborhood interactions. It is expected that our method of proof also
carriers over to this case.
\end{rem}

\begin{rem}
Although Theorem \ref{thm:KPOD} settles the blow-up of \eqref{eq:NSE} more or less satisfactorily,
the proof itself (albeit simple) gives little information
on the transfer of energy mechanism in the model. On the other hand, the previous proofs on the blow-up of KP
model \emph{do} respect
the details of the dynamics and give some insight of the cascade mechanism. In light of this,
it is still desirable to give a
more "dynamic" proof in this flavor. After all, one of the main reasons for
studying the dyadic models is to understand  energy cascade and even turbulence transport.
\end{rem}

Before the proof of Theorem \ref{thm:KPOD}, we state
two direct corollaries which simply correspond to cases $\nu=0$ and
$\beta=0$, respectively.

\begin{cor}[The full model with no diffusion]
\label{cor:KPO}Consider the equations (\ref{eq:KPO}) with parameters
$\l=2$, $\alpha=1$, and $\beta\geq0$, and fix $s>1/3$ together
with $(2^{s}-2^{1-2s})/(1+2^{1-3s})>\beta$. Then for every nonnegative
initial data (that is, $a_{j}(0)\geq0$ for all $j$), there is finite-time
blow up in $H^{s}$-norm.
\end{cor}

\begin{rem}
Here and below (in Corollary \ref{cor:D}), the blow-up of $H^s$-norm is again
understood as that $\| a(t) \|_s^2$ is not locally integrable on $[0,\infty)$.

\end{rem}

For example, when $\beta<6/5$, every nonnegative initial data blow-up
in the $H^{1}$-norm.
\begin{cor}[KP with diffusion]
\label{cor:D} Consider the equations (\ref{eq:NSE}) with $\l=2$,
$\alpha=1$, $\beta=0$, $\nu=1$, and $\gamma<1/3$. For every $s>1/3$,
let $\theta=\theta({s,\gamma})$ be a constant such that
\begin{align*}
&-\frac 43 <\theta<2(s-1), \\
&-\frac 43 < \theta <-4\gamma.
\end{align*}
There exists a constant $C=C(s,\gamma,\theta)>0$, such that once
the initial data $(a_j(0))_{j=0}^{\infty} \in H^s$ satisfy
\[
\sum_{j=0}^{\infty}2^{ j(\theta+1)}a_{j}(0)>C,
\]
there is finite-time blow up in $H^{s}$-norm. \end{cor}
\begin{proof}[Proof of  Theorem \ref{thm:KPOD}]
We fix some $s>1/3$ and assume towards contradiction that
$||a(t)||_{s}^2$ is locally integrable\footnote{This in turn would
imply that one can freely interchange summation and differentiation
in the argument below. Alternatively, one can recast the equations
into integral (in time) formulation and justify passing the limit
under the integral.} on $[0,\infty)$.
 By setting $b_{j}(t):=\l^{j}a_{j}(t)$, we simplify
the equation as follows:
\begin{equation}
\frac{db_{j}}{dt}=\big(\lambda^{2}b_{j-1}^{2}-b_{j}b_{j+1}\big)+\beta\big(\l b_{j-1}b_{j}-\l^{-1}b_{j+1}^{2}\big)-\l^{2\gamma j}b_{j}.\label{eq:NSE-1}
\end{equation}
Then we consider the sum
\[
A(t):=\sum_{j=0}^{\infty}b_{j}^{2}w^{-j}
\]
where $w>1$ is a constant to be optimized later. We observe that
if $w^{-1}\leq\l^{-2}\cdot2^{2s}$, then $A(t)$ is also integrable since $A(t)\leq||a(t)||_{s}^{2}$
for all $t \geq 0$. Then we consider the quantities
\[
\frac{d}{dt}(b_{j}w^{-j})=(\l^{2}b_{j-1}^{2}w^{-j}-b_{j}b_{j+1}w^{-j})+
\beta(\l b_{j-1}b_{j}w^{-j}-\l^{-1}b_{j+1}^{2}w^{-j})-\l^{2\gamma j}b_{j}w^{-j}
\]
and sum them over all $j\geq0$. By the Cauchy-Schwartz inequality, the infinite sum appearing on the right hand side
is bounded in absolute value by $\mathrm{const}\cdot (A(t)+ \sqrt{A(t)}) $, and therefore the sum can be rearranged
whenever $A(t)$ is finite.
 We therefore obtain:
\begin{eqnarray}
\frac{d}{dt}(\sum_{j=0}^{\infty}b_{j}w^{-j}) & = & \l^{2}\sum_{j=1}^{\infty}b_{j-1}^{2}w^{-j}-\sum_{j=0}^{\infty}b_{j}b_{j+1}w^{-j}\label{eq:blowup}\\
 & + & \beta\l\sum_{j=1}^{\infty}b_{j-1}b_{j}w^{-j}-\beta\l^{-1}\sum_{j=0}^{\infty}b_{j+1}^{2}w^{-j}-\sum_{j=0}^{\infty}\l^{2\gamma j}b_{j}w^{-j}.\nonumber
\end{eqnarray}
We note in advance that again by the Cauchy-Schwartz inequality,
\begin{equation}
(\sum_{j=0}^{\infty}b_{j}w^{-j})^{2}\leq(\sum_{j=0}^{\infty}w^{-j})(\sum_{j=0}^{\infty}b_{j}^{2}w^{-j})=\frac{A(t)}{1-w^{-1}}\label{eq:A}
\end{equation}
holds. Then first four terms on the right hand side of (\ref{eq:blowup})
can be estimated as follows:
\begin{eqnarray*}
 &  & \l^{2}w^{-1}\sum_{j=1}^{\infty}b_{j-1}^2 w^{-{j-1}}
 -w^{1/2}\sum_{j=0}^{\infty}(b_{j}w^{-\frac{j}{2}})(b_{j+1}w^{-\frac{j+1}{2}})\\
 &  & +\beta\l w^{-1/2}\sum_{j=1}^{\infty}(b_{j-1}w^{-\frac{j-1}{2}})(b_{j}w^{-j})
 -\beta\l^{-1}w\sum_{j=0}^{\infty}b_{j+1}^2 w^{-{(j+1)}}\\
 & \geq &  (\lambda^2 w^{-1} -w^{1/2}-\beta\l w^{-1/2}-\beta \lambda^{-1} w
 ) \cdot A(t).
\end{eqnarray*}
Regarding the last term, we have
\begin{eqnarray*}
-\sum_{j=0}^{\infty}\l^{2\gamma j}b_{j}w^{-j} & = & -\sum_{j=0}^{\infty}(b_{j}w^{-\frac{j}{2}})(\l^{2\gamma j}w^{-\frac{j}{2}})\geq-(\sum_{j=0}^{\infty}b_{j}^{2}w^{-j})^{\frac{1}{2}}(\sum_{j=0}^{\infty}\l^{4\gamma j}w^{-j})^{\frac{1}{2}}\\
 & = & -\frac 1 {1-\lambda^{4\gamma} w^{-1}} \cdot(\sum_{j=0}^{\infty}b_{j}^{2}w^{-j})^{\frac{1}{2}}\\
 & \geq & - \frac 1 {1-\lambda^{4\gamma} w^{-1}} (\eta A(t)+\frac{1}{4\eta}),
\end{eqnarray*}
where $\eta>0$ is a constant. Here we have assumed that $\l^{4\gamma}w^{-1}<1$.
Adding above two estimates together with (\ref{eq:A}), we conclude
for some $C_{2}>0,$
\[
\frac{d}{dt}(\sum_{j=0}^{\infty}b_{j}w^{-j})\geq C_{1}\cdot(\sum_{j=0}^{\infty}b_{j}w^{-j})^{2}-C_{2}.
\]
Therefore, if we have $C_{1}>0$ and $\sum b_{j}(0)w^{-j}>\sqrt{C_{2}/C_{1}}$, then
$\sum b_{j}w^{-j}$ will not be locally integrable on $[0,\infty)$, which is a contradiction
to the fact that $A(t)$ is locally integrable. To this end, we need
\[
C_{1}=(1-w^{-1}) \cdot\Bigl(\l^{2}w^{-1}-w^{1/2}-\beta\l w^{-1/2}-\beta\l^{-1}w-\frac {\eta} {1-\lambda^{4\gamma} w^{-1}}
\Bigr)>0.
\]
But at the expense of choosing $\eta$ sufficiently small, it is enough
to have
\begin{align*}
\l^{2}w^{-1}-w^{1/2}-\beta\l w^{-1/2}-\beta\l^{-1}w>0.
\end{align*}
Therefore, to conclude the proof, we need:
\begin{align}
w^{-1} & \le \l^{-2}2^{2s} \notag \\
w^{-1} & <  \l^{-4\gamma} \notag \\
\l^{2}w^{-1} - w^{1/2}&>\beta(\l w^{-1/2}+\l^{-1}w), \label{w_ineq}
\end{align}
where $\l=2$. Assuming for the moment that $\beta=0$, we have $\l>w^{3/4}$
from the last inequality which gives restrictions $s>1/3$, $\gamma<1/3$.
On the other hand, it is clear now that once we have $s>1/3,\gamma<1/3$,
we can choose $w$ in a way that for small $\beta>0$, all the above three
inequalities are satisfied.
\end{proof}

\begin{proof}[Proof of Corollary \ref{cor:KPO}]
Since $\nu=0$, the second inequality of \eqref{w_ineq} is not needed.
One can just choose $w^{-1} =2^{2(s-1)}$ and this gives
\begin{align*}
\frac {2^{s}-2^{1-2s}}{1+2^{1-3s}}>\beta.
\end{align*}
\end{proof}
\begin{proof}[Proof of Corollary \ref{cor:D}]
Since $\beta=0$, the conditions on $w^{-1}$ in \eqref{w_ineq} take the form
\begin{align*}
&2^{-\frac 43} <w^{-1} \le 2^{2(s-1)}, \\
&2^{-\frac 43}<w^{-1}<2^{-4\gamma}.
\end{align*}
Denoting $w^{-1}=2^{\theta}$ then yields the result.
\end{proof}

Let us close by presenting a few conjectures which would complement
or generalize regularity and blow-up results currently known. We first
explain the result of \cite{MR2746670}: recall that we have already
mentioned their dissipation of energy result. But they also proved
the existence of so-called ``self-similar solutions'', which are
natural analogues of the fixed point in the forced case. To be specific,
consider the forced KP equations:
\begin{align*}
\frac{d} {dt} a_{j}(t) & =  \lambda^{j}a_{j-1}^{2}(t)-\l^{j+1}a_{j}(t)a_{j+1}(t),\,\,\,(j\geq1)\\
\frac{d}{dt}a_{0}(t)  &= -\l a_{0}(t)a_{1}(t)+f_{0},
\end{align*}
where $f_{0}>0$ is a constant. Then it is immediate that there exists
a unique fixed point which have finite energy. This fixed point satisfies
$\bar{a}_{j}=\mathrm{const}\cdot\l^{-j/3}$ so it has finite $H^{s}$-norm
precisely for $s<1/3$. In \cite{MR2337019,MR2600714}, it was established
that this fixed point is the unique global attractor of the dynamics.

When there is no forcing, there does not exist nontrivial fixed points.
However, self-similar solutions are the correct analogues; we define
a solution self-similar if for every $j\geq0$, $a_{j+1}(t)/a_{j}(t)$
is constant in time. From this requirement, it is straightforward
to check that the solution must have the form
\[
a_{j}(t)=\frac{c_{j}}{t-t_{0}}
\]
for some constants $c_{j}$ and $t_{0}>0$ which satisfy the recurrence (for $\lambda=2$)
\[
c_{j}c_{j+1}=2^{-j}c_{j}+c_{j-1}^{2}/2
\]
for all $j\geq0$ with $c_{-1}=0$. The choice of $c_{0}>0$ uniquely
determines the whole sequence and the self-similar solution, modulo
the choice of $t_{0}>0$ which is independent. The hard part is to
show that there exists a value of $c_{0}>0$ (which turns out to be
unique) such that the self-similar solution $a(t)$ has finite energy.
Then it is not hard to see that the self-similar solution satisfies
$c_{j}\sim\mathrm{const}\cdot\l^{-j/3}$. Note this  power-law decay in
$j$ can already be noticed from our proof; the scale $a_{j}(t)\sim\l^{-j/3}$
roughly corresponds to the case where we have equalities in the Cauchy-Schwartz
inequalities used in the proof. Now it is very desirable to show that
the self-similar solutions are the global attractors of the unforced
dynamics. If we believe in the convergence towards self-similar ones,
it is natural to conjecture that in the KP equations, the $H^{s}$-norms
remain finite for all $s<1/3$. This finiteness of smaller Sobolev
norms are partially obtained in the works \cite{MR3057168,MR2844828}.
Also, one can revert all inequalities in our proof and try to get some a
priori estimates on the solution, which look similar to some regularity
results proved in \cite{MR3057168,MR2844828}. Finally, it is tempting
to believe that such self-similar solutions also exist for our equation,
at least when $\beta$ is small. One can write down the recurrence
relation as above but this relation is now more complicated.

\section*{Acknowledgements}
I. Jeong would like to thank his advisor Prof. Ya.G. Sinai for many
helpful discussions. I. Jeong was supported in part by a Samsung Fellowship. D. Li was supported in part by an Nserc
discovery grant.

\bibliographystyle{plain}
\bibliography{dyadic}

\end{document}